\documentclass[a4paper,11pt]{scrartcl}

\usepackage[fleqn]{amsmath}
\usepackage{amssymb}
\usepackage{amsthm}
\usepackage{stmaryrd}
\usepackage{bbm}
\usepackage{hyperref}
\usepackage{mathtools}
\usepackage[numbers]{natbib}
\usepackage{xcolor}

\mathtoolsset{showonlyrefs}

\DeclareMathOperator{\Var}{Var}

\newcommand{\bE}{\ensuremath{\mathbb{E}}}

\newcommand{\bN}{\ensuremath{\mathbb{N}}}

\newcommand{\bP}{\ensuremath{\mathbb{P}}}

\newcommand{\bR}{\ensuremath{\mathbb{R}}}

\newcommand{\bZ}{\ensuremath{\mathbb{Z}}}

\newcommand{\cA}{\ensuremath{\mathcal{A}}}

\newcommand{\cC}{\ensuremath{\mathcal{C}}}

\newcommand{\cF}{\ensuremath{\mathcal{F}}}
\newcommand{\cG}{\ensuremath{\mathcal{G}}}
\newcommand{\cH}{\ensuremath{\mathcal{H}}}

\newcommand{\cR}{\ensuremath{\mathcal{R}}}

\newcommand{\ddx}[1][1]{\ifnum#1=1 \frac{d}{dx} \else \frac{d^{#1}}{dx^{#1}} \fi}
\newcommand{\ddy}[1][1]{\ifnum#1=1 \frac{d}{dy} \else \frac{d^{#1}}{dy^{#1}} \fi}
\newcommand{\ddt}[1][1]{\ifnum#1=1 \frac{d}{dt} \else \frac{d^{#1}}{dt^{#1}} \fi}

\newcommand{\bZd}{\mathbb{Z}^d}

\newenvironment{remark}[1][Remark]{\begin{trivlist}
\item[\hskip \labelsep {\bfseries #1}]}{\end{trivlist}}

\theoremstyle{plain}
\newtheorem{theorem}{Theorem}[section]  
\newtheorem{proposition}[theorem]{Proposition}

\newtheorem{lemma}[theorem]{Lemma}

\theoremstyle{definition}

\theoremstyle{remark}

\numberwithin{equation}{section}

\title{Limit laws for random walks in a dynamic path-cone mixing random environment}

\author{ 
Stein Andreas Bethuelsen
 \footnote{University of Bergen, Department of Mathematics, Bergen, Norway 
 \newline
 Email: stein.bethuelsen@uib.de}  
 \quad
 Florian V\"ollering   
 \footnote{University of Leipzig, Germany
  \newline
 Email: florian.voellering@math.uni-leipzig.de}  
}

\begin{document}

\maketitle

\abstract{We study the asymptotic behaviour of a random walk whose evolution is dependent on the state of an itself dynamically evolving environment. 
Assuming that the environment decorrelates with time by satisfying the "path-cone"-mixing property introduced in \cite[Bethuelsen and V\"ollering]{BethuelsenVolleringRWDRE2016},  we  prove a strong law of large numbers  and large deviation estimates.  Moreover, under a mild assumption on the decay rate of this mixing property, we obtain a functional central limit theorem under the annealed law.}


\tableofcontents


\section{Introduction, the model and main results}


\subsection{Introduction}
Random walks evolving in a random environment have been studied extensively since its origin in \cite{SolomonRWRE1975}. This model on $\mathbb{Z}$ is by now well understood and known to exhibit phenomena that fundamentally distinguishes its behaviour from that of a standard simple random walk.
The same model on higher dimensional lattices has also attracted much interest, but here less is known and several fundamental questions remain open. See for instance \cite{DrewitzRamirezRWRE2013} for an overview.

In recent years, driven by applications in physics and biology, random walks moving in a dynamically evolving environment have gained increasing interest. 
For such models, qualitatively much is by now known under various types of \emph{fast mixing} assumptions on the environment.  In these cases, the environment de-correlates sufficiently fast for the increments of the  random walk to be approximately independent on large time scales. Hence, its behaviour resembles that of a simple random walk.  On the contrary, heuristics and simulation studies indicate that it may have anomalous behaviour when the environment is \emph{slowly mixing} \cite{AvenaThomannRWDREsimulations2012}. In this paper, we focus on the fast-mixing regime and add to the literature new results about the limiting behaviour for a general class of models.

The \emph{cone-mixing} condition introduced in \cite{CometsZeitouniLLNforRWME2004,ZeitouniRWRE2004} for static environments and adapted in  \cite{AvenaHollanderRedigRWDRELLN2011} to the dynamic setting is one characterisation of fast-mixing. 
In  \cite{BethuelsenVolleringRWDRE2016}  we introduced what we in this paper name the \emph{path-cone} mixing condition, which is a weakening of the cone-mixing condition. Therein we proved (among others) the strong law of large numbers for the position of the random walk. The purpose of the current paper is to present improvements to this work,  particularly by deriving  large deviation bounds (see Theorem \ref{thm:RWDRE}) and the scaling to a Brownian motion (see Theorem \ref{thm:RWDRE2}).



In the following subsection we provide a precise definition of the `random walk in dynamic  random environment' model. Our main results on the asymptotic behaviour of the random walk  are then  presented in Subsection \ref{sec:MR}. The remaining sections summarize the proofs of the main results.  


\subsection{The model}

We now introduce in general terms the model of a random walk in a dynamic random environment (RWDRE), partly following the same notation and conventions as in \cite{BethuelsenVolleringRWDRE2016}.  As our \emph{environment} we consider a (random) configuration $\omega \in \Omega$ where $\Omega := E^{\mathbb{Z}^{d}\times\bZ}$ for some finite set $E$ and $d\in \bN$, and we denote by $\bP \in  \mathcal{M}_1(\Omega)$ its law. Here, $\mathcal{M}_1(\Omega)$ denotes the set of probability measures on $(\Omega, \mathcal{F})$ where $\mathcal{F}$ is the standard product $\sigma$-algebra generated by the cylinder events. 
We assume throughout that $\bP$ is measure preserving with respect to translations, that is, for any $(x,t)\in \mathbb{Z}^{d}\times \bZ$,
\begin{align}\label{eq translation invariant}
\bP(\cdot) = \bP(\theta_{x,t} \cdot ),
\end{align}
 where $\theta_{x,t}$ denotes the shift operator $\theta_{x,t} \omega(y,s) = \omega(y+x,s+t)$. 
What we have in mind is that $\bP$ is the path measure of a stochastic process whose state space is $E^{\bZ^d}$, but note that our setup is more general than this.

The \emph{random walk} $(X_t)$ is a discrete-time stochastic process on $\bZ^d$.  We assume that its transition probabilities depend on the state of the environment within a finite region $\Delta \subset \bZd$ around its current location. To be more precise, enumerate the set $\Omega_0:=E^{\Delta}= \{\omega_i\}_{i=1,\dots, K}$ with $K=|E|^{|\Delta|}$. For each $i\in \{1,\dots, K\}$, we assume given a certain prescribed probability distribution on $\bZd$, denoted by $\alpha(i,\cdot)$. 
Then, given $\omega \in \Omega$ and denoting by $o\in \mathbb{Z}^{d}$ the origin, the evolution of $(X_t)$ is such that $P_{\omega}(X_0 = o) =1$ and, for $t\geq 0$,
\begin{align}\label{eq:defRW}
P_{\omega}(X_{t+1} = y+z \mid X_t = y) = \alpha(i, z),\quad  \text{ if } \theta_{y,t} \omega = \omega_i \text{ on } \Delta.
\end{align}
We denote by $\cR := \{ z \in \bZd \colon \alpha(i,z)>0 \text{ for some }i\}$ its range and assume throughout that this is a finite set. 

The law of the random walk  when we have conditioned on the entire environment, $P_{\omega} \in \mathcal{M}_1((\bZ^d )^{\bZ_{\geq0}})$, is called the \emph{quenched} law. We denote its corresponding $\sigma$-algebra by $\mathcal{G}$. 
For $\mathbb{P} \in \mathcal{M}_1(\Omega)$, we denote by $P_{\bP}  \in \mathcal{M}_1 \left( \Omega \times (\bZ^d )^{\bZ_{\geq0}}\right)$ the joint law of $(\omega,X)$, that is,
\begin{equation}
P_{\bP}(B \times A) = \int_{B}P_{\omega}(A) d\bP(\omega), \quad B \in \mathcal{F}, A \in \mathcal{G}.
\end{equation}
 The marginal law of $P_{\bP}$ on $(\bZ^d)^{\bZ_{\geq0}}$ is the \emph{annealed} law of $(X_t)$.


\subsection{Main results}\label{sec:MR}
Before presenting our main results about the asymptotic behaviour of $(X_t)$, we first introduce the necessary notation and recall the "path-cone"-mixing property introduced in \cite{BethuelsenVolleringRWDRE2016}. For this let, for $t \in \bN$, 
\begin{align}
\cR_t :=\{ x \in \bZd \colon x= \sum_{i=1}^t y_i, y_i \in \cR\}
\end{align} be all points in $\bZd$ that $X_t$ in principle can attain. Further, let 
\begin{align}\label{eq:coneDT}
\cC(l) := \{ (x,t) \in \bZd \times \bZ \colon x \in \cR_{t} + \Delta  \text{ and } t \geq l\}
\end{align} and denote by $\cF_{\cC(l)}$ the sub-$\sigma$-algebra of $\cF$ generated by the cylinders in $\cC(l)$. 
Then, the cone-mixing property 
of \cite[Definition 1.1]{AvenaHollanderRedigRWDRELLN2011} 
alluded to in the introduction translates into requiring that
\begin{equation}\label{eq:conemixing}
\sup_{B \in \cF_{\cC(t)}} \sup_{A \in \cF_{\leq 0}} | \bP(B \mid A) - \bP(B)| \rightarrow 0 \text{ as } t \rightarrow \infty,
\end{equation}
where $\cF_{\leq0}$ is the sub-$\sigma$-algebra of $\cF$ generated by the cylinders in $\bZd\times \{\dots,-2,-1,0\}$. 
\begin{remark}
Note that, since the model of  \cite{AvenaHollanderRedigRWDRELLN2011} evolves in continuous-time, they considered cylinders on $\bZ^d \times [0,\infty)$ with inclination $\theta$, requiring the equivalent of \eqref{eq:conemixing} to hold for all $\theta\in (0,\pi/2)$. Since our process $(X_t)$ evolves in discrete-time, it cannot exit $\cC(0)$ and therefore it suffices in our case to consider the cones defined as in \eqref{eq:coneDT}.
\end{remark}

The path-cone mixing property is the weakening of \eqref{eq:conemixing}  by restricting the conditioning to events along possible random walk paths only. To make this more precise, denote by 
 \begin{align}
 &\Gamma_{-k}:= \left\{(\gamma_{-k},\gamma_{-k+1},...,\gamma_0) \colon \gamma_i \in \bZ^d,  \gamma_i-\gamma_{i-1}\in \cR , -k\leq i<0, \gamma_0=o\right\};
 \\ \label{eq Gamma k} &\Gamma_{-\infty} := \cup_{k \in \bN} \Gamma_{-k} ,
 \end{align}
 i.e.\ the set of all possible backwards random walk trajectories from $(o,0)$ of finite length.  Given $\gamma \in \Gamma_{-\infty}$ and 
$\sigma= (\sigma_1,\dots,\sigma_{|\gamma|}) \in \Omega_0^{|\gamma|}$, let
\begin{equation}\label{eq Akm}
A(\gamma, \sigma) := \bigcap_{i=-{|\gamma|}}^{-1} \left\{ \theta_{\gamma_i,i}\omega= \sigma_{-i} \text{ on } \Delta \right\}
\end{equation}
denote a particular  observation of a random walk along the path $\gamma$ and 
\begin{align}
\mathcal{A}(\gamma) := \left\{ \bigcap_{i=-{|\gamma|}}^{-1} \left\{ \theta_{\gamma_i,i}\omega= \sigma_{-i} \text{ on } \Delta \right\}
\colon \sigma \in \Omega_0^{|\gamma|} \text{ and } \bP(A(\gamma,\sigma))>0 \right\}
\end{align} 
the set of all possible observations. Moreover, let 
$\mathcal{A}_{-\infty}:=\bigcup_{\gamma \in {\Gamma_{-\infty}}} \mathcal{A}(\gamma)$. We say that $\bP$ is \emph{path-cone mixing} if $\lim_{t \rightarrow \infty} \widetilde{\phi}(t) =0$, where
\begin{equation}\label{eq:pcmixing}
\widetilde{\phi}(t) :=  \sup_{B \in \cF_{\cC(t)}} \sup_{A \in \cA_{-\infty}} | \bP(B \mid A) - \bP(B)|.
\end{equation}
See  \cite[Section 2]{BethuelsenVolleringRWDRE2016} for a thorough discussion of distributions $\bP$ satisfying \eqref{eq:pcmixing}. These include the path measure of uniquely ergodic interacting particle systems converging to its stationary distribution sufficiently fast. The condition is, however, not restricted to this class and, particularly in high dimensions, one may expect it to hold for a larger class of interacting particles systems having multiple stationary distributions. 
 
 In addition to mixing of the dynamic environment, one form of lack of memory is often pivotal in studies of RWDRE models. 
 We say that the RWDRE model is (uniformly) \emph{elliptic} if there is an $\epsilon>0$ such that either of the two following conditions hold: 
\begin{align}
\label{con:ellipticity}
\begin{split}
&a) \: \inf_{A \in \mathcal{A}_{-\infty}} \inf_{\sigma \in \Omega_0} \bP(\omega = \sigma \text{ on } \Delta \mid A ) > \epsilon. 
\\ &b) \: \sup_{z \in \bZd} \inf_{i \in \{1,\dots,K\}} \{\alpha(i,z) \} >\epsilon.
\end{split}
\end{align}

Note that Condition a) concerns the environment only and implies that, irrespectively of its past behaviour, $(X_t)$ may observe any element of $\Omega_0$ with a probability of at least $\epsilon$. 
Condition b) on the other hand is a condition on the jump kernel of $(X_t)$ and does not concern the environment.  

The path-cone mixing property was utilised in \cite{BethuelsenVolleringRWDRE2016} to prove, among others, that $(X_t)$ satisfies the Strong Law of Large Numbers (SLLN), assuming that $(X_t)$ satisfies the ellipticity condition b). 
We next present an extension of this result.

\begin{theorem}\label{thm:RWDRE}
Consider an elliptic and path-cone mixing RWDRE model on $\bZd$. 
\begin{itemize}
\item[i)] 
There exists a  $v \in \bR^d$ such that
\begin{equation}\label{eq:SLLN} \lim_{t \rightarrow \infty} \frac{X_t}{t} = v \quad P_{\bP}\text{-almost surely} 
\text{ and in } L^1.
\end{equation}
\item[ii)] 
For every $\epsilon >0$ there exists constants $C,c>0$ such that 
\begin{equation}\label{eq:LDBs}
P_{\bP} \left( \left| \frac{X_t}{t} - v \right| > \epsilon \right) \leq Ce^{-ct}, \quad \text{ for all } t \in \bN,
\end{equation}
with $v\in \bR^d$ as in \eqref{eq:SLLN}.
\end{itemize}
\end{theorem}

Theorem \ref{thm:RWDRE}i) extends \cite[Corollary 1.4]{BethuelsenVolleringRWDRE2016} by relaxing the ellipticity assumption for the SLLN 
and to hold in $L^1$. 
Moreover, the Large Deviation Bounds (LDBs) of  Theorem \ref{thm:RWDRE}ii) are new and extend results obtained in  \cite{AvenaHollanderRedigRWDRELDP2009,BethuelsenHeydenreich2017} 
to environments satisfying far less stringent space-time mixing assumptions.  Note also that in \cite{CamposDrewitzRamirezRassoulSeppalainenQLDPRWDRE2013} a large deviation principle  was established under the quenched law under rather general assumptions on the dynamic environment. Particularly, since the rate function in the quenched setting is upper bounded by annealed large deviation bounds, by Theorem \ref{thm:RWDRE}ii), the corresponding rate function has a unique zero at $x=v$. 

We do not expect the assumptions for the SLLN  to be sharp. In particular, restricted to models on $\bZ$ with nearest neighbour jumps, a SLLN was recently obtained in \cite{BlondelHilarioTeixeiraUniform2020} under a seemingly less restrictive mixing assumption.  Therein they also obtained LDBs, but with a polynomial decay in \eqref{eq:LDBs}. 
See also  \cite{Allasia2024} for an extension of the SLLN to models with  finite range jumps. 

To conclude the annealed functional central limit theorem (aFCLT) we are in need of a slightly stronger ellipticity property. We say that the RWDRE model is (uniformly) \emph{strongly elliptic} if either condition a) above holds or it satisfies condition $b')$ given by: 
 \begin{align} b') \: \exists \; \epsilon>0\: \colon \inf_{z \in \cR} \inf_{i \in \{1,\dots,K\}} \{ \alpha(i,z)\} >\epsilon. \end{align}

\begin{theorem}\label{thm:RWDRE2}
Consider a path-cone mixing and strongly elliptic RWDRE model on $\bZd$. 
\begin{itemize}
\item[i)] If there is $\alpha \in \bR^d$ such that the  variance of $\alpha \cdot X_n$ 
 diverges, then there is a  slowly varying function  $h\colon [0,\infty) \rightarrow \bR$ such that 
\begin{equation}
\left( \frac{\alpha \cdot X_{\lfloor nt\rfloor}- \alpha\cdot v\lfloor nt\rfloor}{\sqrt{nh(n)}} \right)_{t \in [0,1]} \overset{P_{\bP}}{\implies} (B_t)_{t \in [0,1]}
\end{equation}
where $\overset{P_{\bP}}{\implies}$ denotes convergence in distribution under the annealed law $P_{\bP}$ as $n\rightarrow \infty$ with respect to the Skorohod topology and  $(B_t)$ is a standard Brownian motion on $\bR$.

\item[ii)]  If $\lim_{t \rightarrow \infty} \widetilde{\phi}(t)t^{2+\delta}=0$ for some $\delta>0$, then furthermore there is a $\sigma \in (0,\infty)$ such that statement of i) holds with $h(n)= \sigma$.

\item[iii)] If $\lim_{t \rightarrow \infty} \widetilde{\phi}(t)t^{2+\delta}=0$ for some $\delta>0$ and for every  $\alpha \in \bR^d \setminus \{o\}$  the  variance of $\alpha \cdot X_n$  
 diverges, then 
 \begin{equation}
\left( \frac{ X_{\lfloor nt\rfloor}- v\lfloor nt\rfloor}{\sqrt{n}} \right)_{t \in [0,1]} \overset{P_{\bP}}{\implies} (B_t^{\Sigma})_{t \in [0,1]}
\end{equation}
where   $(B_t^{\Sigma})$ is a  Brownian motion on $\bR^d$ with covariance matrix $\Sigma>0$.

\end{itemize}
\end{theorem}

Previously the aFCLT has been proven in \cite{comets2005gaussian} for random walks in static random environments (note also the recent improvements \cite{Guerra1,Guerra2}), which was adapted  to the dynamic setting in \cite[Chapter 3]{AvenaPhD}. The mixing assumptions needed therein are more restrictive than cone-mixing and hence, Theorem \ref{thm:RWDRE2} provides an extension of these works. The aFCLT was also obtain in \cite{RedigVolleringRWDRE2013} under a mixing assumption similar to cone-mixing. Lastly, we note that the aFCLT  follows from the methodology developed in  \cite{DolgopyatKellerLiverani2008} for Markovian environments when combined with the results of \cite[Corollary 1.6]{BethuelsenVolleringRWDRE2016}, under the assumption that the environment is path-cone mixing with $ \widetilde{\phi}(t)$ decaying at an exponential rate.

We believe that the assumption  in Theorem  \ref{thm:RWDRE2} that  the variance of $(X_n)$ diverges is redundant in most cases, but refer to Section  \ref{sec:var} for a more thorough discussion of this matter. We also believe that the assumption on $ \widetilde{\phi}(t)$ in statement ii) and iii) can be considerably relaxed. In particularly, one might expect that the local environment process considered in Section \ref{LERW} should  mix  at the same rate as $ \widetilde{\phi}(t)$ in which  
 case
a decay at order $\log(t)^{-(2+\delta)}$ for some $\delta>0$ would be sufficient. 
The main theorems might also hold beyond the modelling setting of this paper. However, the assumptions that $E$ and $\cR$ are finite come natural in the setting of path-cone mixing environments and are used at several places in the proofs. 

Having obtained the aFCLT in Theorem \ref{thm:RWDRE2}, it is also natural to ask if the same statement holds with respect to the quenched law for $\bP$-a.e.\ realization of the dynamic environment. General quenched central limit theorems have been obtained e.g.\ in \cite{DolgopyatKellerLiverani2008,DolgopyatLiverani2009} under more stringent mixing assumptions than path-cone mixing.  We leave the extension of these to our setting  for future research.


\section{The local environment process}
\label{LERW}

Inspired by \cite{DolgopyatLiverani2009}, we study the so-called \emph{local environment process}. 
To introduce this process precisely, we first recall that 
 the so-called \emph{environment process} is the process $(\omega_t^{EP})$   on $\Omega$ 
  given by
   \begin{equation} (\omega_t^{EP})=(\theta_{(X_t,t)}\omega), \quad t \in \bZ_{\geq0}. \end{equation}

This process was the main object of study in our previous work \cite{BethuelsenVolleringRWDRE2016} and contains information about the entire environment as seen from the walker. The \emph{local environment process} restricts to local information about the environment and is the process $(\xi_t)= (\xi_t^{(1)},\xi_t^{(2)})$ on $\Xi=(\Omega_0, \cR)$ 
 given by
\begin{align}\label{def:LEP}
&\xi_t^{(1)}(x) = \omega_t^{EP}(x,0), \quad x\in \Delta;
\quad &\xi_t^{(2)} := X_{t+1}-X_{t}.
\end{align}
Thus, $(\xi_t^{(1)})$ is the projection of the environment process onto $\Delta$, that is, the environment the random walk needs to observe in order to determine its future behaviour, whereas $(\xi_t^{(2)})$ is simply the jump taken by the random walk at time $t$. Particularly, note that $X_n= \sum_{t=1}^n \xi_t^{(2)}$.

 In the following, we write $P_{-k}$ to denote the law of $(\xi_t)$  shifted by time $t=-k$, that is, the law of $(\hat{\xi}_t)_{t \geq -k}$ where 
$\hat{\xi}_t = \xi_{t+k}$, $t \geq -k$. We also denote by $\cH$ the induced $\sigma$-algebra on $\Xi^{\bZ}$. For $i\leq j$, we write $\cH_{[i,j]}$ for the  induced $\sigma$-algebra containing events on $\Xi^{[i,j]}$ and write $\cH_{\geq t}$ for $\cH_{[t,\infty)}$. 

Now, consider a cylinder event $A \in \cH_{[-k,-1]}$. That is, similar to \eqref{eq Akm},
 for some $\gamma \in \Gamma_{-k}$ and $\sigma \in \Omega_0^{|\gamma|}$, we have that $A = \bigcap_{i=-k}^{-1} \left\{ \hat{\xi}_i^{(1)} =\gamma_{-i},  \hat{\xi}_i^{(2)} =\sigma_{-i}  \right\}$.  
Then it follow by  \cite[Theorem 3.1]{BethuelsenVolleringRWDRE2016} (see particularly the 
last line of the proof) that  $P_{-k}(\cdot | A)$ equals in law to the process $(\xi_t)$ with initial distribution given by $P_{\bP(\cdot | A(\gamma,\sigma))}$. 
With this in mind, for $t \in \bZ_{\geq 0}$, consider the uniform mixing quantity 
\begin{align}\label{eq:LEPmixing}
\widehat{ \phi}(t) = \sup_{A_0,A_1 \in \cA} \sup_{B \in \cH_{\geq t}} | P_{\bP(\cdot \mid A_0)}(B) - P_{\bP(\cdot \mid A_1)}(B)|.
 \end{align}

 In Section \ref{proofs} we provide fairly general arguments on how mixing properties of the local environment process imply asymptotic properties of the random walk $(X_t)$. We summarize our findings in the following proposition.

\begin{proposition}\label{prop:fLEPtRW}
Consider an elliptic RWDRE model on $\bZd$ satisfying  $\lim_{t \rightarrow \infty} \widehat{\phi}(t) = 0$. 
\begin{enumerate}
\item[i)] The law $\frac{1}{k} \sum_{i=1}^{k} P_{-(i)}( \cdot )$ converges weakly to some law $\nu$ on $(\Xi^{\bZ},\cH)$ which is trivial on the tail $\sigma$-algebra $\cH_{\infty} \coloneqq \cap_{s\geq 0} \cH_{\geq s}$. Moreover, $\nu$ and $P_{\bP}$  agree on $\cH_{\infty}$. 
\item[ii)]  $(X_t)$ satisfies the SLLN of Theorem \ref{thm:RWDRE}i) 
with $v=\nu(\xi_1^{(2)})$.
 \item[iii)]  $(X_t)$ satisfies the LDBs of Theorem \ref{thm:RWDRE}ii).
\item[iv)]  If the model is strongly elliptic and  
$\liminf_{n \rightarrow \infty} \Var_{\nu}[\theta \cdot X_n] = \infty$ for some $\theta \in \bR^d$, then $(\theta \cdot X_t)$ satisfies the aFCLT of Theorem \ref{thm:RWDRE2}i).  
Furthermore, if $\sum_{t \geq 1} \sqrt{\widehat{\phi}(2^t)}<\infty$, then 
 $\lim_{n \rightarrow \infty} n^{-1} \Var_{\nu}(\alpha \cdot X_n)$ exists and is strictly positive. 
Lastly,  $(X_t)$ satisfies the aFCLT of Theorem \ref{thm:RWDRE2}iii) whenever $\liminf_{n \rightarrow \infty} \Var_{\nu}[\theta \cdot X_n] = \infty$ for all $\theta \in \bR^d \setminus \{o\}$ for some $\Sigma_{\nu}>0$.
\end{enumerate}
\end{proposition}

Thus, in order to prove asymptotic properties of $(X_t)$ it is sufficient to control the mixing properties of the local environment process in terms of \eqref{eq:LEPmixing}. The next theorem shows that the latter can be controlled by requiring mixing properties of the underlying environment in terms of the path-cone mixing property of \eqref{eq:pcmixing}.

\begin{theorem}\label{thm:main2} 
Consider an elliptic RWDRE model on $\bZd$ 
and write $\epsilon=e^{-c_0}>0$ for the constant in \eqref{con:ellipticity}. Then, for any $c_1>c_0$,  
\begin{equation}\label{eq:MD}
 \widehat{\phi}(t) \leq \widetilde{\phi}( c_1 \log (t)) + 2 \exp(- t^{1-c_0/c_1}/c_1\log(t)).
 \end{equation} 
\end{theorem}

The proof of Theorem \ref{thm:main2} is presented in the following section and is based on the $\epsilon$-coin trick introduced in  \cite{CometsZeitouniLLNforRWME2004,ZeitouniRWRE2004}. However, unlike the approach therein via the construction of an approximate regeneration time, we control the mixing properties of the local environment process directly. This leads, in our opinion, to a simplification of the proof. Moreover, it yields an estimate on the decay of mixing as seen in \eqref{eq:MD} which is interesting in its own. 

\begin{proof}[Proofs of Theorem \ref{thm:RWDRE} and Theorem \ref{thm:RWDRE2}]
By Theorem \ref{thm:main2}, if $\lim_{t \rightarrow \infty} \widetilde{\phi}(t)=0$, then we also have that $\lim_{t \rightarrow \infty} \widehat{\phi}(t)=0$. Hence, we can invoke Proposition \ref{prop:fLEPtRW}i)-ii) to conclude Theorem \ref{thm:RWDRE}. To conclude Theorem \ref{thm:RWDRE2} we note that if $\lim_{t \rightarrow \infty} \widetilde{\phi}(t)t^{2+\delta}=0$, then  Theorem \ref{thm:main2} yields that $\sum_{t \geq 1} \sqrt{\widehat{\phi}(2^t)}<\infty$ and we can invoke Proposition \ref{prop:fLEPtRW}iv). 
\end{proof}


\section{Proofs}\label{proofs}

Here we provide the proofs of the auxiliary results presented in the previous subsection, Proposition \ref{prop:fLEPtRW} and Theorem \ref{thm:main2} respectively. In the final subsection we discuss sufficient conditions for divergence of the variance.

\subsection{Proof of Theorem \ref{thm:main2}}\label{sec:pfmain2}

Consider an elliptic RWDRE model on $\bZd$, $d\geq1$. We now provide the details of the $\epsilon$-coin trick, where $\epsilon>0$ refers to the constant provided by the ellipticity assumption. Informally, this bit of $\epsilon$-probability will allow us to, at least for a finite time window, decouple the random walk from the environment. 

Assume first that Condition b) holds and fix a $z \in \bZd$ satisfying $\inf_i \alpha(i,z)>\epsilon$. Further,  let $U=(U_t)_{t\geq 1}$ be an i.i.d.\ sequence of $Unif(0,1)$-random variables. For a fixed realisation of the environment $\omega$  and $U$,  similarly to  \eqref{eq:defRW}, we consider the random walk $(\tilde{X}_t)$ such that $P_{\omega,u}(\tilde{X}_0 =o)=1$ and, for $t\geq 0$,
\begin{align}\label{eq:defRWellip}
P_{\omega,u}(\tilde{X}_{t+1} = x+y \mid X_t = x) = 1_{u\leq \epsilon}  1_{y=z}  +  1_{u >\epsilon} \frac{(\alpha(i,y) - \epsilon1_{y=z})}{1-\epsilon} 
\end{align}
$ \text{if } \theta_{x,t} \omega = \omega_i \text{ on } \Delta$. 
Note that, when averaging over $U$, the law of $\tilde{X}$ agrees with that of $X$.  
Further, from $\omega$ and $(\tilde{X}_t)$, we can construct the corresponding local environment $(\tilde{\xi}_t)$ as in \eqref{def:LEP}, which now depends on the sequence $U$, but again whose law equals that of $(\xi_t)$ when averaging over $U$. In particular, for proving Theorem \ref{thm:main2} it is sufficient to control the process $(\tilde{X}_t)$. 
For this, the following is a key lemma, where  we denote by
\begin{equation}\label{def:trick1}
\tau_{n} \coloneqq \inf \{ t \geq  n \colon U_s \leq \epsilon \: \forall \: s = t- n+1, \dots, t\}, \quad  n\in \bN,
\end{equation}
that is, the first time that the sequence $(U_t)$ successively has taken values below $\epsilon$  over a time period of length $n$. 

\begin{lemma}\label{lem:trickh1}
Assume that the RWDRE model satisfies Condition b). Then, for any $n \in \bN$, 
\begin{equation}\label{eq:trickh1} 
\sup_{A_0,A_1  \in \cA} \sup_{B \in \cH_{\geq n}} | P_{\bP(\cdot \mid A_0)}(B \mid  \tau_n = n) - P_{\bP(\cdot \mid A_1)}(B \mid  \tau_n = n)| \leq \widetilde{\phi}(n).
\end{equation}
\end{lemma}

\begin{proof}[Proof of Lemma \ref{lem:trickh1}]
Fix $n \in \bN$, $A_0,A_1 \in \cA$ and $B \in \cH_{\geq n}$. Denote by $\widehat{\bP}$ a coupling of  $\bP(\cdot \mid A_0)$ and $\bP(\cdot \mid A_1)$  
satisfying $\widehat{\bP}( \omega^{(\bP( \cdot \mid A_0))} \neq \omega^{(\bP(\cdot \mid A_1))} \text{ on } \cC(n) ) \leq \widetilde{\phi}(n)$, e.g.\  the optimal coupling in the sense of total variation distance on $\cF_{\cC(n)}$. 
Next, extend this coupling to include two independent sequences, $(U_t)$ and $(V_t)$, of i.i.d.\ $Unif(0,1)$ random variables. Then, consider the random walks $(\tilde{X}_{t}^{(\bP( \cdot \mid A_0))})_{t\geq0}$ and $(\tilde{X}_{t}^{(\bP(\cdot \mid A_1))})_{t\geq0}$ in environment $\omega^{(\bP( \cdot \mid A_0))}$ and $\omega^{(\bP(\cdot \mid A_1))}$, respectively, constructed to  satisfy \eqref{eq:defRWellip}, where $(U_t)$ provide the common noise and where both processes apply $V_t$ for determining the jump in case $U_t >\epsilon$. Thus, the two processes make the same jump whenever they observe the same local environment. By construction,  conditioned on $\{\tau_n =n\}$, their paths necessarily agree for the first $n$ steps and, moreover, during this time no information about the true environment is revealed.  
Hence, we have that,  
\begin{align}
\label{derivation1}
&| P_{\bP(\cdot \mid A_0)}(B \mid  \tau_n = n) - P_{\bP(\cdot \mid A_1)}(B \mid  \tau_n = n)|
\\ \leq  & \widehat{\bP} ( \tilde{\xi}^{(\bP( \cdot \mid A_0))}_{t} \neq  \tilde{\xi}^{(\bP(\cdot \mid A_1))}_{t} \text{ for some } t\geq n \mid \tau_n=n) 
\\  \leq &\widehat{\bP}(\omega^{(\bP( \cdot \mid A_0))} \neq \omega^{(\bP(\cdot \mid A_1))} \text{ on } \cC(n) \mid  \tau_n =n),
\\ =  \label{derivation4}&\widehat{\bP}(\omega^{(\bP( \cdot \mid A_0))} \neq \omega^{(\bP(\cdot \mid A_1))} \text{ on } \cC(n)),
\end{align}
where the last equality follows since $\tau$ is independent of the environment. Since this holds for any $A_0,A_1 \in \cA$ and $B \in \cH_{\geq n}$, the claim of the lemma follows.
\end{proof}

\begin{proof}[Proof of Theorem  \ref{thm:main2} under Condition b)]
Let $n \in \bN$ and $A_0,A_1  \in \cA$. 
Then,  for any $B \in \cH_{\geq t}$ with $t>n$, by the triangular inequality, we have that 
 \begin{align*}
\left| P_{\bP(\cdot \mid A_0)}(B) - P_{\bP(\cdot \mid A_1)}(B)\right| \leq 
&\sum_{s=n}^{t}  \left|  P_{\bP(\cdot \mid A_0)}(B \mid  \tau_n = s) - P_{\bP(\cdot \mid A_1)}(B \mid  \tau_n = s)\right| P_{\bP}(\tau_n=s) \\ + &2P_{\bP}(\tau_n>t). 
\end{align*}
 Further, by a time shift and utilising Lemma \ref{lem:trickh1}, we note that
\begin{align}
\label{eq:decaybound}
\left|  P_{\bP(\cdot \mid A_0)}(B \mid  \tau_n = s) - P_{\bP(\cdot \mid A_1)}(B \mid  \tau_n = s)\right| \leq \tilde{\phi}(n).
\end{align}
Consequently, combining these two bounds and that $P_{\bP}(\tau_n \leq t) \leq 1$, we have that
\begin{equation}\label{eq:boundphi}
\left| P_{\bP(\cdot \mid A_0)}(B) - P_{\bP(\cdot \mid A_1)}(B)\right|  \leq 
\widetilde{\phi}(n) + 2P_{\bP}(\tau_n>t).
\end{equation}
Since the sequence $(U_s)$ is i.i.d., by a comparison with a geometric distribution, 
\begin{equation}\label{eq:boundphih}
P_{\bP}(\tau_n>t) \leq (1-\epsilon^n)^{\lfloor t/n \rfloor}. 
 \end{equation}
Using that $1-x\leq \exp(-x)$ and by setting $n=n(t) = c_1\log(t)$ with $c_1>c_0$ where $\epsilon=e^{-c_0}$ it follows that
 $P_{\bP}(\tau_n>t) \leq \exp(- t^{1-c_0/c_1}/c_1\log(t))$. By inserting this bound into \eqref{eq:boundphi}  we conclude the proof. 
 \end{proof}


We now turn to the proof of Theorem  \ref{thm:main2} assuming Condition a) to hold for which we argue by using a slightly differently construction. 
Particularly, we introduce an additional element, denoted by $?$, and let $\tilde{\Omega}_0=\Omega_0 \cup \{?\}$. 
Moreover, similar to the above construction, let  $(U_t)_{t\geq 0}$ be an independent sequence with  $U_t \sim Unif[0,1]$.  Then, for $A \in \cA$, $n \in \bN$, and given the sequence $(u_t)$ of $(U_t)$, we construct a process $(\tilde{\xi}_t)$ on $\tilde{\Omega} \times \cR$ iteratively as follows. 
Firstly, at time $t=0$, if $u_0 \leq \epsilon$, we set $\tilde{\xi}_0^{(1)} = ?$ and otherwise $\tilde{\xi}_0^{(1)} \in \Omega_0 = \{\omega_1,\dots,\omega_K\}$ and sampled according to the probabilities
\begin{equation}\label{eq:EllipticityA0}
 \frac{\bP( \omega \equiv \omega_i  \text{ on } \Delta \mid A) -\epsilon/K}{1-\epsilon}, \quad  i = 1,\dots, K. 
\end{equation}
Then, given the state of  $\tilde{\xi}_0^{(1)}$,  we sample $\tilde{\xi}_0^{(2)}$  according to the probabilities of the corresponding jump function, i.e.\ 
 $\bP(\tilde{\xi}_0^{(2)} = y) = \alpha(i,y)$ if $\tilde{\xi}_0^{(1)} = \omega_i \in \Omega_0$ and $\bP(\tilde{\xi}_0^{(2)} = y) = \alpha(?,y)$
with $\alpha(?,y) =  K^{-1}\sum_{i=1}^K \alpha(i,y)$ if $\tilde{\xi}_0^{(1)} =?$.

Note that, by construction, when averaging over $U_0$, the law of $\tilde{\xi}_0^{(2)}$ agrees with that of $\xi_0^{(2)}$ under $P_{\bP(\cdot \mid A)}$.
Moreover, at least in principle, we can also obtain  $\xi_0$ using this construction. Indeed, given $\tilde{\xi} \in \tilde{\Omega}_0 \times \cR$, we can sample $\xi_0$ according to the probabilities
\begin{align}\label{eq:EllipA}
\begin{split}
\bP( \xi_0 = (\omega_i, y) \mid \tilde{\xi} = (\tilde{\omega}_j,z) ) &=  1_{\tilde{\omega}_j \in \Omega_0} 1_{\xi_0=(\tilde{\omega}_j,z)}
\\ &+  1_{\tilde{\omega}_j=?}1_{y=z} \bP(\theta_{y,1}\omega = \omega_i \mid A) \frac{\alpha(i,y)}{\sum_{l=1}^K \alpha(l,y)},
\end{split}
\end{align}
which,  when averaging over $\tilde{\xi}$ and $U$ yields a process with the same law as the original $(\xi_t)$ process under $P_{\bP(\cdot \mid A)}$.

For $1\leq t \leq \tau_n^{(u)}-n$ with $\tau_n^{(u)} = \inf \{ t \geq  n \colon u_s \leq \epsilon \: \forall \: s = t- n+1, \dots, t\}$ we iterate the above procedure. More precisely, 
 if $u_t \leq \epsilon$, we set $\tilde{\xi}_t^{(1)} = ?$ and otherwise $\tilde{\xi}_t^{(1)} \in \Omega_0$. In the latter case its precise state depends on the realisation of $(\tilde{\xi}_s)_{s=0,\dots,t-1}$ and is sampled (similar to \eqref{eq:EllipticityA0}) according to the probabilities
\begin{equation}
 \frac{\bP \left( \theta_{\sum_{s=0}^{t-1} \tilde{\xi}_s,t} \omega  \equiv \omega_i \text{ on } \Delta \mid A \cap A_t((\tilde{\xi}_s)_{s=0,\dots,t-1})  \right) -\epsilon/K}{1-\epsilon}, \quad i=1,\dots, K.
\end{equation}
Here, $A_t((\tilde{\xi}_s)_{s=0,\dots,t-1}$  corresponds to the event that $(\tilde{\xi}_s)_{s=0,\dots,t-1}$ equals the particular realisation of $\tilde{\xi}_s$ sampled prior to time $t$ and 
$\bP( \theta_{\sum_{s=0}^{t-1} \tilde{\xi}_s,t} \omega  \equiv \omega_i \text{ on } \Delta \mid A \cap A_t(\tilde{\xi}) \} $ is the induced probability on $\Omega_0$, obtained
 in a similar vein as in \eqref{eq:EllipA} in a iterative fashion. 
Then, given the state of $\tilde{\xi}_t^{(1)}$, we sample the state of $\tilde{\xi}_t^{(2)}$  in the same way as at time $t=0$.

Now, for times $t\geq \tau_n^{(u)}-n+1$, we change approach, mimicking that as used under Condition b).  Firstly, we  sample the future environment $\omega \in \Omega$ on $\bZd \times \bZ_{\geq \tau_n^{(u)}}$ according to $P_{\bP}( \cdot | A \cap A_{\tau_n^{(u)}-n}((\tilde{\xi}_s)_{s=0,\dots,\tau_n^{(u)}-n}))$, but shifted such that $(\sum_{s=0}^{\tau_n^{(u)}-n} \tilde{\xi}_s, \tau_n^{(u)}-n)$ is the new space-time origin. Moreover, for each $t \in [\tau_n^{(u)}- n+1, \tau_n^{(u)}]$, we set $\tilde{\xi}_t^{(1)}=?$ and  sample  $\tilde{\xi}_t^{(2)}$ independently  according to the probabilities of the corresponding jump function $\alpha(?,\cdot)$.
 Then,   letting $\tilde{X}_{\tau_n^{(u)}} = \sum_{t=0}^{\tau_n} \tilde{\xi}_t^{(2)}$, we sample $(\tilde{X}_t)_{t\geq \tau_n^{(u)}+1}$ according to \eqref{eq:defRWellip} and from this we construct the corresponding local environment process  $(\tilde{\xi}_t)_{t \geq \tau_n^{(u)}-n +1}$ by using  \eqref{def:LEP}. 
What we have obtained by the above is a construction of a process $(\tilde{\xi}_t)$ such that, when averaging over $U$, the law of $(\tilde{\xi}_t^{(2)})$ agrees with that of $(\xi_t^{(2)})$ under $\bP(\cdot \mid A)$. Moreover, the law of $(\tilde{\xi}_t^{(1)})$ agrees with that of $(\xi_t^{(2)})$ for all $t\geq \tau_n^{(u)}$. 

\begin{proof}[Proof of Theorem  \ref{thm:main2} under Condition a)]
Let $n \in \bN$ and $A_0,A_1  \in \cA$. 
We argue similarly as in the proof under Condition b), but now with the above new construction of $(\tilde{\xi}_t)$. This still depends on an i.i.d.\ sequence $(U_t)$ of   $Unif(0,1)$-random variables and so we can define $\tau_n$ just as in \eqref{def:trick1}. 
Further, for any $B \in \cH_{\geq t}$  with $t>n$, by the triangular inequality, we have that 
 \begin{align*}
\left| P_{\bP(\cdot \mid A_0)}(B) - P_{\bP(\cdot \mid A_1)}(B)\right| \leq 
&\sum_{s=n}^{t}  \left|  P_{\bP(\cdot \mid A_0)}(B \mid  \tau_n = s) - P_{\bP(\cdot \mid A_1)}(B \mid  \tau_n = s)\right| P_{\bP}(\tau_n=s) \\ + &2P_{\bP}(\tau_n>t). 
\end{align*}
Now, to control the term within the sum, we couple the two processes $(\tilde{\xi}_t^{(A_0)}$ and $(\tilde{\xi}_t^{(A_1)})$ having marginals $P_{\bP(\cdot \mid A_0)}$ and $P_{\bP(\cdot \mid A_1)}$ respectively using the new construction. For this, $(U_t)$ provide the common noise and, prior to time $\tau_n-n$, an additional sequence  $(V_t^{(1)},V_t^{(2)})$ of i.i.d.\ $Unif(0,1)$ random variables, common to both processes, is applied for determining $\tilde{\xi}_t$  in case $U_t >\epsilon$ in the above iterative scheme. Thus, whenever $U_t<\epsilon$, the two processes agree at time $t$ and this happens independently of the evolution of the processes prior to this time. Moreover, in the event that $\tau_n=s$, we sample the future environment for the two processes according to the optimal coupling,  in the sense of total variation distance on $\cF_{\cC(n)}$, of 
$P_{\bP}( \cdot | A_0 \cap A_{s-n}((\tilde{\xi}_t^{(A_0)})_{t=0,\dots,s-n}))$ and $P_{\bP}( \cdot | A_1 \cap A_{s-n}((\tilde{\xi}_t^{(A_1)})_{t=0,\dots,s-n}))$ respectively, shifted such that the origin corresponds to $(\sum_{t=0}^{s-n} \tilde{\xi}_t^{(A_0)}, s-n)$ and $(\sum_{t=0}^{s-n} \tilde{\xi}_t^{(A_1)}, s-n)$ for the two environments, which we denote by  
$\omega^{(A_0)}$ and $\omega^{(A_1)}$. Then, by construction, we have that 
\begin{align}
&\left|  P_{\bP(\cdot \mid A_0)}(B \mid  \tau_n = s) - P_{\bP(\cdot \mid A_1)}(B \mid  \tau_n = s)\right|
\\ \leq & \widehat{\bP} \left( \tilde{\xi}_t^{(A_0)} \neq \tilde{\xi}_t^{(A_1)} \text{ for some } t \geq s  \mid \tau_n = s \right)
\\ \leq & \widehat{\bP} \left( \omega^{(A_0)} \neq \omega^{(A_1)} \text{ on } \cC(n) \mid \tau_n = s \right)
\\ \leq & \phi(n)
\end{align}
Here the second to last inequality follows since $\tilde{\xi}_t^{(A_0)}=\tilde{\xi}_t^{(A_1)}$ for each $t= s-n,\dots,n$ and so the corresponding random walks necessarily are at the same location at time $s$, and the last inequality holds by construction of the coupling $\widehat{\bP}$. Thus, \eqref{eq:decaybound} holds, and from this we conclude the proof by exactly the same arguments preceding \eqref{eq:decaybound}  as in the proof of Theorem   \ref{thm:main2} under Condition b).
\end{proof}
 

\subsection{Proof of Proposition \ref{prop:fLEPtRW}}

By standard compactness arguments, there exists a sequence $(t_k)_{k\geq\bN}$ along which the Cesaro limits of the local environment process converges weakly towards a stationary distribution, say $\nu$, defined on $(\Xi^{\bZ},\cH)$. That is, for any local event $B \in \cH_{[m,n]}$,
\begin{equation}\label{eq Cesaro}
\lim_{k \rightarrow \infty} \frac{1}{t_k} \sum_{i= \max (1,-m)}^{t_k} P_{-i}( B ) = \nu(B).
\end{equation}
and $\nu(B) = \nu(\theta_s B)$ for any $B \in \cH$, where $\theta_s$ denotes the shift operator on $\Xi^{\bZ}$ such that $\theta_s\xi_t = \xi_{t+s}$. 

Our first result is an easy, but robust comparison between $\nu$ and the local environment process that will be important to the following analysis.

\begin{lemma}\label{lem comparison nu}
Let $\nu$ be any limiting measure as in \eqref{eq Cesaro}. Then, for any event $A \in  \cH_{\leq -1}:= \cH_{(-\infty,-1]}$, on $\cH_{\geq0}$ 
with $\nu(A)>0$,
\begin{equation}
\inf_{i \in \bN} \inf_{A_0 \in \cH_{[-i,-1]}} P_{-i}(\cdot \mid A_0) \leq  \nu(\cdot \mid A) \leq \sup_{i \in \bN} \sup_{A_1 \in \cH_{[-i,-1]}} P_{-i}(\cdot \mid A_1)
\end{equation}
\end{lemma}

\begin{proof}
Assume without loss of generality that $A$ is a cylinder event in $\cH_{[-m,-1]}$ with $m\geq 1$. Then, 
\begin{align}
\nu(B \mid A) &= \lim_{k \rightarrow \infty} \frac{1}{t_k} \sum_{i=0}^{t_k-m} P_{-(i+m)}( B \cap A)/  \nu(A).
\\ &=\lim_{k \rightarrow \infty} \frac{1}{t_k} \sum_{i=0}^{t_k-m} P_{-(i+m)}( B \mid A)  P_{-(i+m)}( A)/ \nu(A) .
\\ &\geq \inf_{i \in \bN} \inf_{A_0 \in \cA_{-i}^{-1}} P_{-i}(B \mid A_0) \lim_{k \rightarrow \infty} \frac{1}{t_k} \sum_{i=0}^{t_k-m}  P_{-(i+m)}( A)/ \nu(A) 
\\ &=\inf_{i \in \bN} \inf_{A_0 \in \cA_{-i}^{-1}} P_{-i}(B\mid A_0)
\end{align}
since $ \lim_{k \rightarrow \infty} \frac{1}{t_k} \sum_{i=0}^{t_k-m}  P_{-(i+m)}( A)= \nu(A)$ by assumption.  This yields the lower bound. The proof of the upper bound is analogous. 
\end{proof}

Thus, in order to control the stationary distribution $\nu$ it is sufficient to control the conditional measures $P_{-i}(\cdot \mid A)$ for $A \in  \cH_{[-i,-1]}$. Next we take advantage of this observation to provide sufficient conditions for transferring mixing properties from the local environment process to $\nu$.  Before stating these conditions, we recall that any stationary distribution on $(\Xi^{\bZ},\cH)$, say $\nu$, is \emph{ergodic} if all translation invariant events, i.e.\ events $A \in \cH$ such that $\theta_s A=A$, have either $\nu$-measure $0$ or $1$.

\begin{proof}[Proof of Proposition \ref{prop:fLEPtRW} i)]
Let $\nu^*$ be any limiting measure as in \eqref{eq Cesaro}. By Lemma \ref{lem comparison nu}, if $\lim_{t \rightarrow \infty} \widehat{\phi}(t)=0$, then also $\lim_{t \rightarrow \infty} \phi(t)=0$, where 
\begin{equation}\label{phi mixing}
\phi(k) := \sup_{A  \in \cH_{\leq-1}} \sup_{B \in \cH_{\geq k}} |\nu(B \mid A) - \nu(B)|.
\end{equation} 
That is, $\nu^*$ is so-called $\phi$-mixing. Since this applies to any limiting measure defined as in \eqref{eq Cesaro} they all necessarily $\phi$-mixing. 
Now, as is well known, $\phi$-mixing implies triviality on $\cH_{\infty}$ and hence ergodicity. Moreover, all such limiting measures agree since they agree on $\cH_{\infty}$ and are ergodic. Thus, the measure given by \begin{equation}\label{eq ams}  \nu(\cdot)  \coloneqq \lim_{k \rightarrow \infty} \frac{1}{k} \sum_{i=1}^{k} P_{-(i)}( \cdot ) \end{equation} is well defined. 
In the terminology of \cite{GrayBook2009}, this means that $P_{\bP}$ is asymptotically mean stationary (AMS) with $\nu$ as its stationary mean. Particularly, by \cite[Corollary 7.6]{GrayBook2009}, $\nu$ and $P_{\bP}$ agree on $\cH_{\infty}$.
\end{proof}

\begin{proof}[Proof of Proposition \ref{prop:fLEPtRW} ii)]
Since the assumption of Proposition \ref{prop:fLEPtRW} i) is fulfilled, we have that  $\nu$ given by the limit \eqref{eq ams} is well defined, ergodic and agree with $P_{\bP}$ on $\cH_{\infty}$. 
Thus
\begin{equation}
\lim_{t \rightarrow \infty} t^{-1}X_t = \lim_{t\rightarrow \infty} t^{-1} \sum_{i=1}^t \xi_i^{(2)}
\end{equation}
converges a.s.\ both under $P_{\bP}$ and $\nu$ to the same limit $v=\nu(\xi_1^{(2)})$, where the latter follows since $\nu$ is ergodic. 
Moreover, as concluded in the proof of Proposition \ref{prop:fLEPtRW} i), we have that $P_{\bP}$ is AMS with respect to $\nu$ and therefore, by e.g.\ \cite[Theorem 8.1]{GrayBook2009}, the convergence also holds in $L_1$.
\end{proof}

\begin{proof}[Proof of Proposition \ref{prop:fLEPtRW} iii)]
Recall from the proof of Proposition \ref{prop:fLEPtRW} i) that $\nu$ is $\phi$-mixing. This implies that also the projection of $(\xi_t)$ onto the second coordinate is stationary and $\phi$-mixing under $\nu$.  Thus, the statement  follows by  \cite{SchonmannMixing1989} in this stationary case by noting that $X_t=\sum_{i=1}^t \xi_i^{(2)}$ and since the range of $(X_t)$ is bounded. 
 By a detailed look at the proof in \cite{SchonmannMixing1989} it is evident that the stationarity assumption can be replaced by the property that $\lim_{t \rightarrow \infty} \widehat{\phi}(t) =0$. 
 For completeness, we now present the details, assuming without loss of generality that $v=0$.

Let $N$ and $L$ be (large) natural numbers that we determine later. For $q,j,k \geq 1$, let
\begin{align}
&Y_j^{(q)} = \frac{1}{L} \sum_{i=1}^L \xi^{(2)}_{i+NL(j-1) + (q-1)L}, 
\quad &Z_{k}^{(q)} = \frac{1}{k} \sum_{j=1}^k Y_j^{(q)}
\end{align}
and note that 
$\frac{1}{NLk} X_{NLk} = \frac{1}{N} \sum_{q=1}^N Z_k^{(q)}$. 
Thus, for any $\epsilon>0$ fixed, we have that
\begin{align}\label{eq:ldb1}
P_{\bP}( \| \frac{1}{NLk} X_{NLk} \| > \epsilon/2) \leq \sum_{q=1}^N P_{\bP}(\| Z_k^{(q)} \|>\epsilon/2)
\end{align}
Further, by the (exponential) Markov  inequality, we have that 
\begin{equation}
P_{\bP}(\| Z_k^{(q)} \|>\epsilon/2) \leq e^{-k\epsilon/2} E (\exp(\sum_{j=1}^k \| Y_j^{(q)}\| ) )
\end{equation}
The gist of the proof is now to choose $N$ and $L$ large in such a way that the latter expectation term does not blow up too fast in $k$. For this,  fix $\delta< \frac{1}{2}(e^{\epsilon/2}-1)$ and set $N$ such that, for any $q\geq 1$, 
\[  |E_{P_{\bP}}( \exp( \| Y_{j+1}^{(q)} \| ) \mid Y_1^{(q)} \dots, Y_j^{(q)} ) - E_{P_{\bP}}( \exp (\| Y_1^{(q)} \|) )| \leq\  \widehat{\phi}(N) e^R\leq  \delta, \]
 where $R = \max ( \|x\| \colon x \in \cR )$. 
Further, since $t^{-1} X_t \rightarrow v$ in $L_1$ by Proposition \ref{prop:fLEPtRW}ii), 
we may tune $L$ large  such that (recall that we consider $v=0$) 
 \begin{equation}
 E_{P_{\bP}}(\exp(\frac{1}{L} \|X_L \| )) =E_{P_{\bP}}(\exp(\| Y_1^{(1)} \|)  < 1+\delta.
 \end{equation} 
  Consequently,    using the triangular inequality,   we have that
\begin{align}
E_{P_{\bP}}( \exp( \| Y_{j+1}^{(q)} \| ) \mid Y_1^{(q)} \dots, Y_j^{(q)} ) 
\leq 1+ 2\delta.
\end{align}
In particular, this implies that 
\begin{equation}
E_{P_{\bP}} (\exp(\sum_{j=1}^k \| Y_j^{(q)}\| ) ) \leq (1+2\delta)^k.
\end{equation}
By our choice of $\delta$, we hence have that $P_{\bP}(\| Z_k^{(q)} \|>\epsilon/2)$ decays exponentially in $k$, and by \eqref{eq:ldb1} the claim of the corollary follows for times $t=NLk$, $k\geq 1$. To extend the exponential decay to arbitrary $k \geq 1$ is standard, see e.g.\ \cite[Page 482]{BrycSmolenski1993}. 
\end{proof}

The proof of Proposition \ref{prop:fLEPtRW} iv) relies on the following lemma.

\begin{lemma}\label{lem:ac}
If the model is strongly elliptic and $\lim_{t \rightarrow \infty}  \widehat{\phi}(t) =0$, then $\nu$ and $\bP$ are mutually absolutely continuous on  the sub-$\sigma$-algebra $\cG_{\geq0} \subset \cH_{\geq0}$ only concerning events of the second coordinate $(\xi^2)$ (i.e.\ the jump steps of $(X_t)$). 
\end{lemma} 

\begin{proof}[Proof of Lemma \ref{lem:ac}]
 By the strong ellipticity property,  for any $k \geq 0$, the measures $\nu$ and $P_{\bP}$ are mutually absolutely continuous on the sub-$\sigma$-algebra $\cG_{[0,k]} \subset \cH_{[0,k]}$ restricted to events of the second coordinate $(\xi^2)$.
Therefore, by \cite[Theorem 5]{EngelbertShiryaev1980}, we have that $\nu$ and $P_{\bP}$ are mutually absolutely continuous on $\cG_{\geq0}$ if and only if 
\begin{equation}\label{eqMAC}
\nu \left(\limsup_{n \rightarrow \infty} Z_n< \infty  \right) = 1 \:  \text{ and } \: P_{\bP} \left(\limsup_{n \rightarrow \infty} W_n < \infty  \right) =1.
\end{equation}
where $Z_n \coloneqq  \frac{d\nu|\cG_{[0,n]}}{dP_{\bP}|\cG_{[0,n]}}$ and $W_n \coloneqq \frac{dP_{\bP}|\cG_{[0,n]}}{d\nu|\cG_{[0,n]}}$.  
Since the events in \eqref{eqMAC}   are in the tail $\sigma$-algebra $\cH_{\infty}$, by Proposition \ref{prop:fLEPtRW} i) this is equivalent to
 \begin{equation}
P_{\bP} \left(\limsup_{n \rightarrow \infty} Z_n < \infty  \right) = 1 \:  \text{ and } \: \nu \left(\limsup_{n \rightarrow \infty}W_n < \infty  \right) =1.
\end{equation} 
 Now note that $(Z_n, \cG_n,\nu)$ and $(W_n, \cG_n,P_{\bP})$ form non-negative martingales. Therefore, applying Proposition \ref{prop:fLEPtRW} i)  once more, these converge to $Z_{\infty}$  and $W_{\infty}$, respectively, both $\nu$-a.s.\  and $P_{\bP}$-a.s. Further, by this and Fatou's lemma, we find that
\begin{equation}
E_\nu \left(\limsup_{n \rightarrow \infty} \frac{d\nu|\cG_{[0,n]}}{dP_{\bP}|\cG_{[0,n]}}   \right) \leq 1
  \:  \text { and } \:E_{P_{\bP}} \left(\limsup_{n \rightarrow \infty} \frac{dP_{\bP}|\cG_{[0,n]}}{d\nu|\cG_{[0,n]}}  \right) \leq 1.
  \end{equation} From this we conclude that \eqref{eqMAC} indeed holds. 
\end{proof}


\begin{proof}[Proof of Proposition \ref{prop:fLEPtRW} iv)] 
As concluded in the proof of Proposition \ref{prop:fLEPtRW} i), the process $(\xi_t)$ under $\nu$ is $\phi$-mixing. Thus, again since $(\xi_n^{(2)})$ is bounded it follows by \cite[Theorem 2.1]{Peligrad1990} that, as soon as $\liminf_{n \rightarrow \infty} \bE_{P_\nu}[(\alpha \cdot X_n- \bE_{P_\nu}(\alpha \cdot X_n))^2] = \infty$, the process $(\alpha \cdot X_n)$ satisfies an aFCLT under $\nu$ in the form of Theorem \ref{thm:RWDRE2} i). Moreover, \cite[Theorem 1]{Bradley1981} implies that under the additional assumption on the mixing rate, the statement of Theorem \ref{thm:RWDRE2} ii) holds with $\sigma=\lim_{n \rightarrow \infty} n^{-1} \Var_{\nu}(\alpha \cdot X_n)$. From this the conclusion of Theorem \ref{thm:RWDRE2} iii) holds by applying Cram\'er-Wolds device.

In order to conclude the aFCLT 
under $P_{\bP}$, by \cite[Corollary 3]{Zweimuller2007}, it is sufficient to show that $P_{\bP}$ is absolutely continuous to $\nu$ on  the sub-$\sigma$-algebra $\cG_{\geq0} \subset \cH_{\geq0}$ only concerning events of the second coordinate $(\xi^2)$ (i.e.\ the jump steps of $(X_t)$). This is the statement of Lemma \ref{lem:ac} and from which we conclude the proof.
\end{proof}

\subsection{Divergence of the variance}\label{sec:var}

As shown in \cite[Theorem 2.1]{Ibragimov75}, for stationary real-valued $\phi$-mixing processes either $\sup_{n} \Var(X_n) <\infty$ or $\Var(X_n) = n h(n)$ for some slowly varying function $h$. In the literature on central limit theorems for weakly dependent random variables it is therefore common practice to assume that the variance diverges to infinity \cite{Bradley1981}.  General arguments for proving this divergence seem rather scarce with \cite{Aizenman90} and \cite{Chatterjee} being notable exceptions. The method of the former was applied successfully to prove the aFCLT in \cite[Theorem 3.5]{RedigVolleringRWDRE2013} which, when properly adapted to our model setting yields that the variance grows linearly in $n$ when $\widetilde{\phi}$ decays exponentially fast. 

We believe that the assumption on the variance of $X_n$ in Proposition  \ref{prop:fLEPtRW} iii)-iv) and Theorem \ref{thm:RWDRE2} for most (if not all) path-cone mixing and strongly elliptic RWDRE models is redundant. For instance, with $\alpha \in \bR^d \setminus \{o\}$, a simple application of Chebychevs inequality implies that
\begin{equation}
 \Var(\alpha \cdot X_n)  \geq \epsilon^2 P( |\alpha \cdot X_n- \alpha \cdot nv |>\epsilon) 
\end{equation}
from which it follows that $\limsup_n P( |\alpha \cdot X_n- \alpha \cdot nv |>\epsilon(n))>0$ for some $\epsilon(n)\rightarrow \infty$ suffices. Under certain restrictions on the dimension,  by the arguments of  \cite{PeresPopovSousiRTRW2013}, the latter holds even without any mixing requirements.  More precisely, assuming that $\{ y \in \bZ^d \colon \alpha(i,y)>0\}$ span $\bZ^d$ for each $i = 1,\dots, K$, and reasoning as in the proof of  \cite[Proposition 1.4]{PeresPopovSousiRTRW2013}, we have that there is some universal constant $C>0$ such that
\begin{equation}
P(\| X_n- nv \|>\epsilon)  \geq 1- C \epsilon^d/ n^{d/2-K+1}.
\end{equation}
Hence, we have  $\liminf_{n \rightarrow \infty}  \Var_{\nu}(\alpha \cdot X_n) = \infty$ for any  $\alpha \in \bR^\setminus \{o\}$  whenever  $d> 2(K-1)$.


\subsection*{Acknowledgement}
The authors thanks an anonymous referee for valuable comments and suggestion, and acknowledge support from the Deutsche Forschungsgemeinschaft (DFG, German Research Foundation) through the scientific network \emph{Stochastic Processes on Evolving Networks}.


\bibliographystyle{plainnat}

\end{document}